\documentclass[12pt]{article}

\usepackage{amsmath, amsfonts, amssymb}

\usepackage[utf8]{inputenc}
\usepackage[T1]{fontenc}
\usepackage{todonotes}
\usepackage{comment}
\usepackage{enumitem}
\usepackage[left=2cm, right=3cm, top=2cm, bottom=2cm, bindingoffset=0cm, marginparwidth=3cm]{geometry}

\usepackage{hyperref}
\usepackage{float}
\usepackage{graphicx}
\usepackage[outdir=./]{epstopdf}
\graphicspath{ {./} }

\usepackage{amsthm}

\usepackage{tikz}
\usetikzlibrary{arrows,calc}
\tikzset{
>=stealth',
}

\newcounter{theorems}
\newtheorem{prop}[theorems]{Proposition}
\newtheorem{thm}[theorems]{Theorem}

\theoremstyle{definition}

\theoremstyle{remark}
\newtheorem*{ex}{Example}
\newtheorem{rem}[theorems]{Remark}
\newtheorem{lemma}[theorems]{Lemma}
\newtheorem{theoremA}[theorems]{Theorem A}
\newtheorem*{rem*}{Remark}

\makeatletter
\def\blfootnote{\gdef\@thefnmark{}\@footnotetext}
\makeatother

\def\om {\omega }

\def\e {\varepsilon }
\def\d {\delta}
\def\phi {\varphi}
\def\be {\begin{equation}}
\def\ee {\end{equation}}
\def\bt {\begin{thm}}
\def\et {\end{thm}}

\def\NN {\mathbb N}

\DeclareMathOperator{\dist}{dist}

\def \XX {{\mathbb X}}
\newcommand{\diam}{\operatorname{diameter}}
\newcommand{\interior}{\operatorname{interior}}
\def \ca {\mathcal{A}}

\def\NN {\mathbb N}

\begin{document}

\title{Attractors with Non-Invariant Interior and Pinheiro's Theorem A}
\author{Stanislav Minkov\footnote{Brook Institute of Electronic Control Machines, Moscow, Russia; stanislav.minkov@yandex.ru}, \; Alexey Okunev, \ Ivan Shilin \hphantom{1} }
\date{}

\maketitle

\begin{abstract}
This is a provisional version of an article, intended to be devoted to properties of attractor's intertior for smooth maps (not diffeomorphisms). We were originally motivated for this research by Pinhero's Theorem A from his preprint~\cite{P}, and in Section~\ref{s:proof-A} we give a simple and straightforward proof of this result.
\end{abstract}

\section{ Pinheiro's Theorem A}
In this section we quote the statement of Theorem~A from Pinheiro's preprint~\cite{P}. We will give a simple and straightforward proof of this result in Section~\ref{s:proof-A} below.

Let $\XX$ be a compact metric space.
Given a compact %
 set $A$ such that $f(A)=A$ \footnote{In \cite{P} sets with this property are called forward invariant sets, but we prefer another convention: in the text below <<$A$ is forward invariant>> means $f(A)\subset A$.}, define the {\bf\em basin of attraction of $A$} as $$\beta_f(A)=\{x\in\XX\,;\,\omega_f(x)\subset A\}.$$
Following Milnor's definition of topological attractor (indeed, minimal topological 
attractor \cite{M}), a compact forward invariant set $A$
is called  {\bf\em a topological attractor} if
$\beta_f({A})$ is not a meager set and $\beta_f(A)\setminus\beta_f(A')$ is not a meager set   for every compact forward invariant set $A'\subsetneqq A$. According to Guckenheimer \cite{G}, a set $\Lambda\subset\XX$ has {\bf\em sensitive dependence on initial condition} if exists $r>0$ such that $\sup_{n}\diam(f^n(\Lambda\cap B_{\varepsilon}(x)))\ge r$ for every $x\in\Lambda$ and $\varepsilon>0$.

If $\cup_{n\ge 0} f^n(V)=X$ for every open set $V \subset X$, then $f$ is called  {\bf\em  strongly transitive} on $X$.

\begin{theoremA}
\label{thm:A}Let $f:\XX\circlearrowleft$ be a continuous {\it open} map defined on a compact metric space $\XX$. 
If there exists $\delta>0$ such that $\overline{\bigcup_{n\ge0}f^n(U)}$ contains some open ball of radius $\delta$, for every nonempty open set $U\subset\XX$, then there exists a finite collection of topological attractors $A_1,\cdots, A_\ell$ satisfying  the following properties.
\begin{enumerate}
\item $\beta_f(A_j)\cup\cdots\cup\beta_f(A_{\ell})$ contains an open and dense subset of $\XX$.
\item  Each $A_j$ contains an open ball of radius $\delta$ and $A_j=\overline{\interior(A_j)}$.
\item  Each $A_j$ is transitive and $\omega_f(x)=A_j$ generically on  $\beta_f(A_j)$.
\item $\overline{\Omega(f)\setminus\bigcup_{j=0}^{\ell}A_j}$ is a compact set with empty interior.
\end{enumerate}
Furthermore, if $\bigcup_{n\ge0}f^n(U)$ contains some open ball of radius $\delta$, for every nonempty open set $U\subset\XX$, then the following statements are true.
\begin{enumerate}\setcounter{enumi}{4}
\item For each $A_j$ there is a set $\ca_j\subset A_j$ containing an open and dense subset of $A_j$ such that $f(\ca_j)=\ca_j$ and $f|_{\ca_j}$ is strongly transitive.
\item Either $\omega_f(x)=A_j$ for every $x\in\ca_j$ or $A_j$ has sensitive dependence on initial conditions.
\end{enumerate}
\end{theoremA}

\section{A simple continuous map whose attractor has non-invariant interior}
The original statement by Pinheiro \cite[Theorem~A]{P} did not require the map $f$ to be open. Yet the proof implicitly used the invariance of the attractor interior under the dynamics, which is not true in general. %

Here is a counterexample to the original Pinheiro theorem~A.
Let $\XX$ be $[-1,1]\cup 2$, and let $f([-1,1]) = \{2\}$ and $f(2) = 0$. Clearly, $f$ is continuous. Every orbit of $f$ contains the point~$2$, and therefore contains a ball of radius 0.5 centered at~$2$ (this ball coincides with $\{2\}$). The attractor $A_1$ of $f$ is $\{0, 2\}$, with 2 being an interior point of~$A_1$ and 0 being a boundary point. Therefore $A_1 \neq \overline{\interior(A_1)}$, and the statement (2) of Theorem~A is false. Moreover, the interior point 2 of $A_1$ is taken into a boundary point~0.

Counterexamples for path-connected manifolds are also possible.

In this paper we slightly alter the statement of the theorem by requiring $f$ to be open. We will only use this assumption in the proofs of claims~(2) and~(5).

\section{Proof of Theorem~A} \label{s:proof-A}

\subsection*{Lemmas on the $\omega$-limit sets}

Consider a map that takes a point of the phase space into the closure of its positive semi-orbit under~$f$. Denote this map by $\Psi$ and observe that it is lower semicontinuous (w.r.t. the Hausdorff distance between compact subsets of~$\XX$), since finite parts of the forward semi-orbit depend continuously on the initial point. By the semicontinuity lemma~\cite{S}
the set $R$ of continuity points of $\Psi$ is residual in the phase space.

\begin{lemma}\label{lem:d_ball}
In the assumptions of Theorem~A, for any point $x \in R$ the $\omega$-limit set of~$x$ contains an open ball of radius~$\d$.
\end{lemma}
\begin{proof}
Let $U_n = B_{1/n}(x)$ be an open ball of radius $1/n$ centered at~$x \in R$. By assumption,  the closure of the forward orbit of $U_n$ contains an open ball of radius $\d$. Denote the center of this ball by $y_n$. Let $y_0$ be an accumulation point for the sequence $\{y_n\}$. Then a $\d$-ball centered at $y_0$ is contained in the set $\Psi(x) = \overline{\rm{Orb^+(x)}}$. Indeed, if this is not the case and there is a point~$z$ of this ball outside~$\Psi(x)$, then a ball of small radius $\e$ centered at $z$ is disjoint from~$\Psi(x)$. But then for every point $\hat{x}$ in a sufficiently small neighborhood of~$x$ the set $\Psi(\hat{x})$ is $\e/2$-close to $\Psi({x})$ (since $x$ is a continuity point for the map~$\Psi$) and hence does not contain~$z$. But this is in contradiction with the point~$z$ being in  $\overline{\bigcup_{j\ge0}f^{j}(U_{n_k})}$ for a sequence $n_k \to +\infty$. This yields that the set $\Psi(x)$
contains the $\d$-ball centered at~$y_0$. Since the continuity set $R$ is invariant, an analogous statement is true for $f^n(x), \; n \in \NN$. Recall that $\omega(x) = \cap_{n \ge 0} \overline{\rm{Orb^+(f^n(x))}}$. Each set $O_n = \overline{\rm{Orb^+(f^n(x))}}$ contains a $\d$-ball. As above, we take a subsequence of centers of the balls that converges to some point $z_0$ and observe that for a $(\d-\e)$-ball centered at $z_0$ we can find arbitrarily large $n$ such that this ball is contained in $O_n$. Since $O_n$ form a nested sequence, $B_{\d -\e}(z_0)$ is contained in the intersection $\cap_n O_n = \omega(x)$. Hence, there is an open $\d$-ball in $\omega(x)$.
\end{proof}
\begin{lemma}
\label{lem:coinc}
If the interiors of $\om(a)$ and $\om(b)$ have nonempty intersection for $a, b \in R$, then $\om(a) = \om(b)$.
\end{lemma}
\begin{proof}
If the two interiors intersect, then the set $\omega(a) \cap \omega (b)$ contains an open ball~$B$. Since the orbit of $b$ must approach every point of this ball, there is $N$ such that $f^N(b) \in B \subset \om(a)$, and hence for any $n > N$ we have  $f^n(b) \in \omega (a)$, by the invariance of $\omega(a)$, which yields $\omega(b) \subset \omega (a)$. Analogously, $\omega(a) \subset \omega (b)$.
\end{proof}

\subsection*{Proof of claims (1)--(4)}

Let us say that the two points in $R$ are equivalent if the interiors of their $\om$-limit sets intersect (this relation is transitive by Lemma~\ref{lem:coinc}). The set $R$ then splits into a finite number of equivalence classes: indeed, for each class the $\om$-limit set of its point contains an open $\d$-ball and those balls for different classes are disjoint, but one can fit only a finite number of disjoint $\d$-balls into a compact metric space. 

The attractors $A_j$ are exactly the $\om$-limit sets that correspond to these equivalence classes. The basin of each $A_j$ is open. Indeed, let $A_j = \omega(x)$. The set $A_j$ contains a ball, so any point~$y$ close to $x$ will get inside this ball under the iterates of~$f$, by continuity, and so it will be attracted to $A_j$: $\om(y) \subset A_j$. Hence, the basin of $A_j$ is open. Also, since for any $x$ in the residual set $R$ the limit set $\om(x)$ coincides with some $A_j$, a proper subset of any $A_j$ contains $\om(y)$ only for a meager set of $y$-s, so each $A_j$ is a topological Milnor attractor.

Since each $A_j$ contains an open ball, it contains a point $y \in R$. But then $A_j = \om(y)$, and so the forward orbit of $y$ is dense in $A_j$. As the point $y$ is recurrent, this makes the attractor transitive. This, together with the genericity of $R$ in $\beta_f(A_j)$ yields claim 3. Now, we are assuming that $f$ is an open map. This implies that $f(\mathrm{int}(A_j)) \subset \mathrm{int}(A_j)$, and so any $\omega$-limit point for a point $z \in \mathrm{int}(A_j)$ is in $\overline{\mathrm{int}(A_j)}$. Since we can take $z\in R$ with $\om(z) = A_j$, this yields that $A_j$ coincides with $\overline{\mathrm{int}(A_j)}$. This finishes the proof of claim~(2).

Observe that if a non-wandering point is in $\beta_f(A_j)$, it belongs to $A_j$. Indeed, the orbit of this point visits a $\d$-ball inside $A_j$, and hence a small neighborhood of this point is taken inside $A_j$ by some iterate of $f$, say $f^N$. But since the point is non-wandering, there is $K > N$ such that the image of this neighborhood under $f^K$ (the image is contained in $A_j$) intersects the neighborhood itself. This implies that our non-wandering point is accumulated by the points of $A_j$, and so it belongs to $A_j = \overline{A_j}$. Hence, if a non-wandering point does not belong to the union of the attractors $A_j$, it does not belong to the union of $\beta_f(A_j)$, but the complement of the latter union is contained in a closed set with empty interior, by claim~(1), and this implies claim~(4).

\subsection*{Proof of claims (5)--(6)}

Note that it is not very important that the sets $\mathcal A_j$ in claims 5 and 6 of the theorem coincide: we can construct two different forward-invariant subsets $\mathcal A^5_j$, $\mathcal A^6_j$ that contain open and dense subsets of $A_j$ and have the properties from claims 5 and 6 respectively \footnote{One can check that sensitive dependence on initial conditions on a set $S$ is inherited by any residual subset of $S$.} and then take their intersection as $\mathcal A_j$. 

In the rest of the proof, we will call the assumption that for any nonempty open $U$ the union of its images contains a $\d$-ball \emph{the main assumption}.

Fix some $j$ and consider a finite set $\{x_i\}$ such that the union of $\delta / 3$-balls centered at $x_i$ covers $A_j$. Let us denote by $B_i$ the intersections of these balls with $A_j$. %
We will refer to $\{x_i\}$ as the $\delta / 3$-covering for $A_j$.%
Let $C_i = \bigcup_{n\in \mathbb N_0} f^{n}(B_i)$. Each $C_j$ is open with respect to the subset topology on $A_j$: the balls $B_i$ are open in this topology and the restriction $f|_{A_j}$ is open since $f$ is. The attractor $A_j$ is transitive, its dense forward orbit visits each $B_i$, so each $C_i =\bigcup_{n \ge 0} f^{n}(B_i)$ must be dense in~$A_j$.

Let $E = \cap_i (C_i)$ and $\mathcal A^5_j = \cap_{n\in \mathbb N} f^n (E)$. It is clear that $f(\mathcal A^5_j)=\mathcal A^5_j$
The set $E$ is open in~$A_j$ as a finite intersection of open sets $C_i$. By claim~(2) this set has nonempty intersection with $\mathrm{int}(A_j)$, so it contains a small open ball (i.e., the ball is open in $\XX$). Since $E$ is forward invariant ($f(E)\subset E$) %
, it contains the union of the images of this ball, and this union contain a $\d$-ball of the phase space, by the main assumption. The same argument applies to $f(E)$ (recall that $f|_{A_j}$ is open) and every $f^n(E)$. Since the sets $f^n(E)$ form a nested sequence, their intersection also contains a $\d$-ball (one can argue as in the proof of Lemma~\ref{lem:d_ball}). So, $\mathcal A^5_j$ contains a $\d$-ball and, by forward-invariance, all of its images. But there is a point in this ball whose forward orbit is dense in $A_j$, so the union of open $f^n$-images of the ball is open and dense in $A_j$, so $\mathcal A^5_j$ contains an open (w.r.t. the topology of $\XX$) subset which is dense in~$A_j$.

Take any $U$ open in $A_j$. Denote by $V$ the union of the $f^n$-images of~$U$. By claim~(2) and the main assumption, there is a ball of radius $\delta$ in $V$. This ball contains one of the sets $B_i$ (because those have diameter $2\delta/3$ and cover $A_j$), and hence $V$ contains the set $C_j$, and so $V$ contains the whole $\mathcal A^5_j$, which means that $f|_{\mathcal A^5_j}$ is strongly transitive. Also note that, as we showed, $\mathcal A^5_j$ coincides with some~$C_i$.

Now let us prove claim~(6) of the theorem.
We fix some attractor $A_j$ and consider a point in its basin. We will call such a point $r$-punctured if its $\omega$-limit set has empty intersection with some open $r$-ball in $A_j$.

Either there exists an $r$ such that $A_j$ admits a $\d$-covering (with elements contained in $A_j$) that consists of $3r$-punctured points, or not.

In the first case let us take an arbitrary ball $G$ in $A_j$ (i.e., we regard $A_j$ as a metric space and take a ball in it).
The union of its images contains a $\d$-ball (as $G$ contains an open subset of $\interior A_j$ by claim (2)). Then there is a point of our $\delta$-covering in this $\d$-ball, and so the union has a $3r$-punctured point.
Since any preimage of a $3r$-punctured point is also $3r$-punctured, there is a $3r$-punctured point $b$ in $G$. Denote by $B$ a $3r$-ball in $A_j$ such that $B \cap \omega(b)=\varnothing$. There exists an integer $M$ such that for $m>M$ we have that $f^m(b)$ is at a distance at least $2r$ to a center $y$ of $B$ (otherwise $B$ would intersect~$\omega(b)$; note that the ball can have more than one center). Since the attractor is transitive, there is a point $x \in G$ whose forward orbit is dense in $A_j$. Let $N>M$ be a moment of time such that $f^N(x)$ is $r$-close to~$y$. Then $dist (f^N(x), f^N(b))>r$, and we have sensitive dependence on initial condition, with this~$r$.

Now suppose that for any $r$ there is no $\d$-covering of $3r$-punctured points for $A_j$. For ${r=1/n}$ there must be a $\d$-ball in $A_j$ that has no $3/n$-punctured points: otherwise we would construct a $\d$-covering of punctured points. Denote a center of this ball by~$w_n$. Let $w_0$ be an accumulation point for $\{w_n\}$. In a $\d/2$-ball cantered at $w_0$ there are no $r$-punctured points, for arbitrary~$r$; denote this ball by~$P$. For every point in $P$ its forward orbit is dense in $A_j$ --- otherwise there would be $r$-punctured points. Let $\mathcal A^6_j = \bigcup_{k\in \mathbb Z} f^k(P) \cap A_j$. This set satisfies $f(\mathcal A^6_j) = \mathcal A^6_j$ and consists of points whose forward orbits are dense in $A_j$. 

The subset $\hat{\mathcal A}^6_j = \bigcup_{n\in \mathbb N_0} f^{-n}(P) \cap A_j \subset \mathcal A^6_j$ is open and dense in $A_j$.
Indeed, it is open as a union of preimages of open subset $P$ of $A_j$ under continuous $f|_{A_j}$. Suppose that this set is not dense in~$A_j$, that is, there is a ball $V$ disjoint from it. But any point of~$P$ is taken into~$V$ by some iterate of $f$, and then gets back to $P$ (recall that for $z \in P$ the forward orbit is dense), so some points of $V$ are preimages of points in~$P$, i.e., they belong to $\hat {\mathcal {A}^6_j}$. This contradiction shows that $\hat{\mathcal A}^6_j$ is dense in $A_j$.

Thus, $\mathcal A^6_j$ has the properties from claim~(6).

\section{A smooth map with non-invariant attractor interior on a solid torus}

In this section we present an example of a smooth map from a path-connected manifold to itself that has a Milnor topological attractor whose interior is not forward invariant.

\begin{thm}
There exists a smooth map $F$  on a solid torus with a structure of a skew product over the octupling map on a circle $S^1$: $\varphi \mapsto 8 \varphi$ with a Milnor topological attractor whose interior is not forward invariant under~$F$.
\end{thm}
\begin{proof}

The fiber $M$ of our skew product is a two-dimensional disk with radius~3. The skew product itself has the form
\[S^1 \times M \ni (\varphi, x) \stackrel{F}{\mapsto} (8 \varphi, f_{\varphi}(x)),\]
and we refer to $f_\varphi$ as fiber maps. The octupling map on $S^1$ is semi-conjugate via the so-called symbolic coding with the left shift~$\sigma$ on the space $\Sigma^+_8$ of one-sided infinite sequences over the alphabet $\{0, \dots 7\}$. This allows us to encode the fiber maps using these sequences and write $f_\omega$ instead of $f_\varphi$ if $\omega$ is taken to $\varphi$ by the semiconjugacy map.

In our example, the fiber maps $f_\om$ will depend only on the element at the leftmost position in the sequence~$\om$, provided that this element is even ($\{0,2,4, 6\}$). We start indexing with $0$ and write $f_{\om} = f_{\om_0}$; $\om_0$ is the element at position $0$ in the sequence~$\om$. 
The fiber maps for the sequences that start with an odd element in general depend on the whole sequence and are used primarily to make the transition between the ``even'' fiber maps smooth.

Now we describe the properties of the four fiber maps~$f_0, f_2, f_4, f^6_j$. For convenience, we choose a rectangular coordinate system on the disk $M$, such that the point $(0,0)$ is at the center of~$M$. Denote a semi-disk defined by condition $x<0$ by~$D$ and a subset defined by $x<1$ by~$Z$. Now, the three maps $f_2, f_4, f^6_j$ must have the following properties:

\begin{enumerate}[label={\arabic*)}]
   
 \item they are smooth on $M$ and contracting on $Z$ (i.e., there is $\lambda<1$ such that for any $x,y$ $\dist(f_i(x), f_i(y))<\lambda\cdot\dist(x,y)$, for $i=2,4,6$);

 \item $D \subset f_2(D)\cup f_4(D) \cup f^6_j(D)\subset Z$;

 \item $f_i(M \setminus Z) \subset D$ for $i=2,4,6$;

 \item $f_i(Z) \subset Z$ for $i=2,4,6$.

\end{enumerate}

Let the last map $f_0$ just take the whole fiber into a point~$q$ with coordinates~$(2, 0)$.

Now we are prepared to use the Hutchinson lemma for maps $f_2, f_4, f^6_j$:

\begin{lemma} (Hutchinson~\cite{Hutch}, in the form from~\cite{VI})  Consider a metric space $(M,\rho)$ and maps $f_n: M \to M$. Suppose that there exist compact sets $D  \subset Z  \subset M$ such that $f_n(Z) \subset Z$ for all $f_n$, all $f_n$ are contracting on $Z$ and $D \subset \cup f_n(D)$. Then for any open $U$, $U \cap K \neq \varnothing$, there exists a finite word $w=\omega_1\dots \omega_m$ such that $f^{[m]}_w(Z) \subset U$, where $f^{[m]}_w = f_{\omega_1}\circ\dots\circ  f_{\omega_m}$. 
\end{lemma}

The smoothing maps must belong to one of the following two types:

(i) either they take the whole~$M$ into a point at the axis $\{y=0\}$,

(ii) or they are smooth and take $M$ into a subset of~$Z$.

For convenience we describe a tuple of maps with these properties.  

\begin{ex}

Let $f_s$ take $M$ into the point $s=(-1,0)$ and let $a_{[l,r]}$ be a smooth monotonically increasing map from $[l,r]$ to $[0,1]$ whose every derivative vanishes at the endpoints: $a^{(n)}(l)=a^{(n)}(r)=0$.

Now put
\[f_{\varphi}=\left(1-a_{[\frac{3\pi}{4}, \frac{4\pi}{4}]}(\varphi)\right)\cdot f_2 + a_{[\frac{3\pi}{4}, \frac{4\pi}{4}]}(\varphi)\cdot f_4 \quad
\text{ for } \; \varphi \in \left[\frac{3\pi}{4}, \frac{4\pi}{4}\right];\]
\[f_{\varphi}=\left(1-a_{[\frac{5\pi}{4}, \frac{6\pi}{4}]}(\varphi)\right)\cdot f_4 + a_{[\frac{5\pi}{4}, \frac{6\pi}{4}]}(\varphi)\cdot f^6_j \quad
\text{ for } \; \varphi \in \left[\frac{5\pi}{4}, \frac{6\pi}{4}\right].\]

Every $f_{\varphi}$ takes $M$ to itself $M$, because $M$ is convex (and therefore if, for example, $f_2(p) \in M$ and $f_4(p) \in M$, then $(1-a)f_2(p) + a f_4(p) \in M)$. Moreover, $f_{\varphi} (M) \subset Z$, because for any $p$ $x(f_2(p))<1, x(f_4(p))<1$ and then $x((1-a)f_2(p) + a f_4(p))<1$.

For $\varphi \in [\frac{\pi}{4}, \frac{3\pi}{8}]$ let $f_{\varphi}$ be $(1-a_{[\frac{\pi}{4}, \frac{3\pi}{8}]}(\varphi))f_0 + a_{[\frac{\pi}{4}, \frac{3\pi}{8}]}(\varphi) f_s$, and for $\varphi \in [\frac{15 \pi}{8}, 2\pi]$  let $f_{\varphi}$ be $(1-a_{[\frac{15 \pi}{8}, 2\pi]}(\varphi))f_s + a_{[\frac{15 \pi}{8}, 2\pi]}(\varphi) f_0$. All this maps take $M$ into a point $(1-a)q+as$ or $(1-a)s+aq$ respectively, and this points all lie on the axis $y=0$.

At last, for $\varphi \in [\frac{3\pi}{8}, \frac{2\pi}{4}]$ we put $f_{\varphi}=(1-a_{[\frac{3\pi}{8}, \frac{2\pi}{4}]}(\varphi))f_s + a_{[\frac{3\pi}{8}, \frac{2\pi}{4}]}(\varphi) f_2$, and for $\varphi \in [\frac{7\pi}{4}, \frac{15 \pi}{8}]$ we put $f_{\varphi}=(1-a_{[\frac{7\pi}{4}, \frac{15 \pi}{8}]}(\varphi))f^6_j + a_{[\frac{7\pi}{4}, \frac{15 \pi}{8}]}(\varphi) f_s$.  Map $f_{\varphi}$ here is from $M$ to $M$, because $M$ is convex. Moreover, $f_{\varphi} (M) \subset Z$, because for all $p$ $x(f_s(p))<1$ and $x(f_2(p))<1$, and then $x((1-a)f_s(p) + a f_2(p))<1$.

Clearly, all the maps are smooth and we are done.
\end{ex}

\bigskip
\begin{rem}
Since any point in $(M \setminus Z)\cap (y\neq 0)$ has an empty preimage, such points are not in the Milnor topological attractor of~$F$.
\end{rem}

We will need the following auxiliary statement.

CLAIM. A generic sequence $\om \in \Sigma^+_8$ contains infinitely many instances of every finite word of  $0, \dots, 7$. 

\begin{proof}
The space $\Sigma^+_8$  is a Baire space. Given any finite word, the set of sequences that contain this word is open and dense in $\Sigma^+_8$. The intersection $R$ of these sets over all finite words is residual in~$\Sigma^+_8$ and has the required property. Indeed, given a word $w$ and a sequence $\om \in R$, one can find an instance of $w$ in $\om$ simply by the definition of~$R$. Now we take another word $w_1$ such that $\om$ does not contain an instance of the concatenated word $ww_1$ that starts on the left from the right end of the instance of $w$ above. Since $\om \in R$ must contain an instance of $ww_1$, this yields another instance of $w$ on the right from the first one, and we get an infinite series of instances by continuing in the same way.
\end{proof}

Denote the subset $\Sigma^+_8 \times (D \cup \{q\})$ by~$A$.

\begin{prop}
The set $A$ is a subset of the Milnor topological attractor.   \footnote{Note, that in \cite{M} it is proved, that measure Milnor attractor always exists. One can easily modify this proof to check, that Milnor topological attractor also always exists.}

\end{prop}

\begin{proof}
Let $\hat{R} = R \times M$, where $R$ is the residual subset of $\Sigma^+_8$ from the claim. The set $\hat{R}$ is residual in the phase space. Fix any point $r = (\om, x) \in \hat{R}$. Let us prove that the forward $F$-orbit of $r$ is dense in $A$. %
Given $u = (\hat{\om}, x_0) \in A$, we consider its neighborhood $U$ that consists of points whose projections to $\Sigma^+_8$ and $M$ are $\e$-close to the projections of~$u$. 

First assume that $x_0 \in D$ and denote $p = f_2(q)$, $f_2(q) \in D$ by (3). Due to Hutchinson lemma, there are $n$ and $w_p$ such that $f_{w_p}^{[n]}(p)$ is $\e$-close to $x_0$ (and $w_p$ only contains symbols $2,4,6$).
Furthermore, let $w_0$ be an initial word of the sequence $\hat{\om}$ such that any sequence that starts with $w_0$ is $\e$-close to $\hat{\om}$. Let $w$ be the word that consists of a zero, two and the word $w_p$ followed by the word $w_0$.
Since the base component $\om$ of $r$ is in $R$, it contains the word $w$; say the $w_0$-part of this word starts at position $k$.
Observe that the map $f_0$ takes any point of the fiber to $q$, then $f_2$ takes $q$ into $p$ and $f_{w_p}^{[n]}$ takes $p$ into a point that is $\e$-close to $x_0$.
Hence we have $F^k(r) \in U$: we stop iterating when the fiber-component is near $x_0$ and the $\Sigma^+_8$-component begins with $w$ and so is close to $\hat{\om}$.

If $x_0 = q$, we argue analogously, with $w$ being a word of one $0$ followed by~$w_0$. This shows that the orbit of $z \in \hat{R}$ is dense in $A$, and hence $\om(z) \supset A$ and, by the way, $A$ is a transitive set.
Therefore, $A$ is a subset of a Milnor topological attractor.
\end{proof}

Now, any point $(\om, x)\in A$ with $\om_0 = 0$ and $x \in D$ is in $\mathrm{Int}(A)$, but its $F$-image belongs to a boundary of the attractor (because points from $(M \setminus Z)\cap (y\neq 0)$ are not in the attractor). This finishes the proof of the theorem.

It is worth noting, that from density of $\hat{R}$ and from previous result one can conclude, that $F$ is a map with the Pinheiro property: for any open $V$ $\overline{\bigcup_{n\ge0}f^n(V)} \supset A$ contains a fixed ball.

An analogous proof can be done for the measure Milnor attractor, i.e., the likely limit set.

\end{proof}

\section{On generic transitive attractors with nonempty interior}

\begin{thm}
For a generic $C^1$-map $F$ from a manifold $M$ to itself, if $F$ has a transitive closed forward-invariant set~$A$ with nonempty interior, then $A$ coincides with the closure of its interior.
\end{thm}

\begin{proof}
\begin{enumerate}
\item For a generic $C^1$-smooth map $F \colon M \to M$ there is an open and dense set of regular points in~$M$. Indeed, fix some countable dense subset $C$ of $M$. Given a point in this subset, there is an open and dense set in the space of $C^1$-maps of $M$ such that the point is regular for maps in this set. The intersection of these open and dense sets is a residual set that consists of maps for which every point in~$C$ is regular. Clearly, the set of regular points is open by the continuity of the Jacobian.

\item Denote by $R$ the residual set of maps for which regular points form an open and dense set in~$M$. For $F \in R$ for any $N>0$ there is an open and dense set $X_N \subset M$ such that on $X_N \cup F(X_N) \cup \dots \cup F^N(X_N)$ the map $F$ is regular. Indeed, let $Y$ be the set of regular points of $F$, it is open and dense. The set $F^{-1}(Y)$ is open by continuity, let us show that it is dense. Take some open set $U$; restricting to a small open subset we can assume that $F$ is regular on $U$ and, moreover, that $F: U \mapsto F(U)$ is a diffeomorphism. This implies $F^{-1}(Y)$ is dense in $U$. We have proved that $F^{-1}(Y)$ is open and dense; we can show in the same way that all sets $F^{-k}(Y)$, $k>0$, are open and dense too.   

\item Pick open $B \subset A$ and any $a \in A$. Let $U \ni a$ be some neighborhood. By transitivity for some $N>0$ the set $F^{-N}(U)$ intersects $B$.
Denote $C = X_N \cap F^{-N}(U) \cap B$, it is an open set and $F^N$ is regular on it. Thus $F^{N}(C)$ contains an open set. We have $F^{N}(C) \subset A$, as $A$ is forward-invariant. Also, $F^{N}(C) \subset U$. Thus $U$ intersects $\mathrm{Int}(A)$. As this holds for any $U$, this means $a \in \overline{\mathrm{Int}(A)}$. 
\end{enumerate}
\end{proof}

\end{document}